\documentclass[10pt]{amsart}

\usepackage{enumerate, amsmath, amsfonts, amssymb, amsthm, mathtools, thmtools, wasysym, graphics, graphicx, xcolor, frcursive,xparse,comment,bbm,soul}
\usepackage[all]{xy}
\usepackage[colorlinks=true, pdfstartview=FitV, linkcolor=blue, citecolor=blue, urlcolor=blue]{hyperref}

\usepackage{url, hypcap}
\hypersetup{colorlinks=true, citecolor=darkblue, linkcolor=darkblue}
\usepackage{mathrsfs}

\usepackage{tikz}
\usetikzlibrary{arrows,automata}
\usetikzlibrary{patterns,snakes}
\usetikzlibrary{shapes.geometric,calc}
\tikzstyle{arrow}=[thick, ->, >=stealth]
\tikzset{
	edge/.style={->,> = latex'}
}
\usetikzlibrary{calc,through,backgrounds,shapes,matrix}
\usepackage{graphicx}

\setstcolor{red}

\usepackage[letterpaper,margin=1.25in]{geometry}

\definecolor{darkblue}{rgb}{0.0,0,0.7} % darkblue color
\definecolor{darkgreen}{rgb}{0, .6, 0} % darkgreen color
\definecolor{lightblue}{rgb}{0,135,147}

\definecolor{red}{rgb}{1,0,0}

\definecolor{darkred}{rgb}{0.7,0,0} % darkred color
\definecolor{lightgrey}{rgb}{0.7,0.7,0.7} % darkred color
\newtheorem{theorem}{Theorem}[section]
\newtheorem{proposition}[theorem]{Proposition}

\newtheorem{lemma}[theorem]{Lemma}

\theoremstyle{definition}

\newtheorem{example}[theorem]{Example}

\newtheorem{remark}[theorem]{Remark}

\numberwithin{equation}{section}

\definecolor{darkred}{rgb}{0.7,0,0} % darkred color
\newcommand{\defn}[1]{{\color{darkred}\emph{#1}}} % emphasis of a definition

\usepackage[colorinlistoftodos]{todonotes}

\def\edge{\relbar\joinrel\relbar}

\title
  {Holonomy theorem for finite semigroups}
  
\author[J.~Rhodes]{John Rhodes}
\address[J. Rhodes]{Department of Mathematics, University of California, Berkeley, CA 94720, U.S.A.}
\email{rhodes@math.berkeley.edu, blvdbastille@gmail.com}

\author[A.~Schilling]{Anne Schilling}
\address[A. Schilling]{Department of Mathematics, UC Davis, One Shields Ave., Davis, CA 95616-8633, U.S.A.}
\email{anne@math.ucdavis.edu}

\author[P.~V.~Silva]{Pedro V. Silva}
\address[P. V. Silva]{Centro de
Matem\'{a}tica, Faculdade de Ci\^{e}ncias, Universidade do
Porto, R. Campo Alegre 687, 4169-007 Porto, Portugal}
\email{pvsilva@fc.up.pt}

\date{\today}

\begin{document}

\begin{abstract}
We provide a simple proof of the Holonomy Theorem using a new Lyndon--Chiswell length function on the 
Karnofsky--Rhodes expansion of a semigroup. Unexpectedly, we have both a left and a right action on the Chiswell 
tree by elliptic maps.
\end{abstract}

\maketitle

%%%%%%%%%%%%%%%%%%%%%%%%%%%%%%%%%%%%%%%%%%%%%%%%%%%%%%%%%%%%%%
\section{Introduction}
%%%%%%%%%%%%%%%%%%%%%%%%%%%%%%%%%%%%%%%%%%%%%%%%%%%%%%%%%%%%%%

The following problem is at the heart of the global theory of finite semigroups. Let $S,T,U$ be finite semigroups. Consider a 
product of the form $T \star U$ (direct, semidirect, wreath, block, triple -- see \cite{RS.2009}) admitting a projection 
morphism on $U$. When does $S$ divide $T \star U$ (i.e. $S$ is a homomorphic image of some subsemigroup of $T \star U$)? 
The standard technique consists of proving that, for a given generating set $A$ of $S$, each $a \in A$ can be lifted to some 
$\overline{a} \in T \star U$. Examples include the $V \cup T$ proof of the Krohn--Rhodes Theorem (see \cite{RS.2009}) or the alternative 
proof by Zeiger \cite{Zei.1967,Til.1976a,Til.1976b}, called the Holonomy Theorem proof.

Denoting by $\overline{S}$ the subsemigroup of $T \star U$ generated by $\{ \overline{a} \mid a \in A\}$, 
these theorems prove that $S$ is a homomorphic image of $\overline{S}$.  But sometimes it takes decades to discover the exact 
nature of $\overline{S}$. Rhodes' first attempt for the case of the Holonomy theorem was the Rhodes expansion $(S,A)^R$ 
(see \cite{Til.1976a,Til.1976b}) on the way to proving the Fundamental Lemma of Complexity of finite semigroups
(surmorphisms which are one-to-one on subgroups preserve complexity, see \cite{RS.2009}). 

A better version is the right Karnofsky--Rhodes expansion $\mathsf{KR}_{\operatorname{right}}(S,A)$ used in this paper. If we 
denote by $T_n$ the semigroup of transformations of an $n$-set, then, for $n \geqslant 3$, $T_n$ is not a subsemigroup of 
any nontrivial product. So $T_n$ must be expanded by taking preimages to become a subsemigroup of a wreath product 
and is intuitively the ``smallest'' one. We note that $\mathsf{KR}_{\operatorname{left}}(S,A)$ (for left acting semigroups, 
note the right/left reversal) relates to {\em coupling from the past} in Markov chains \cite{ProppWilson.1996,RhodesSchilling.2019}.

In this paper, the Holonomy Theorem is proved via Lyndon--Chiswell theory of semigroups acting on trees (created by Rhodes in~\cite{Rhodes.1991} 
and developed by Rhodes and Silva in \cite{RS.2012}). We use that $\mathsf{KR}_{\operatorname{right}}(S,A)$ acts faithfully as elliptic maps 
on a tree, elliptic maps being an abstraction of wreath products. A key idea is that $\mathsf{KR}_{\operatorname{right}}(S,A)$ has 
many more useful surmorphisms, namely when $\mathsf{KR}_{\operatorname{right}}(S,A)$ is proved to consist of elliptic maps: 
restricting the elliptic maps to all vertices of bounded depth gives a surmorphism.

Another key idea is that the $\mathcal{J}$-order comes explicitly into the definition of the Lyndon--Chiswell function, and 
$\mathsf{KR}_{\operatorname{right}}(S,A)$ carries more information than the Rhodes expansion $(S,A)^R$. This extra 
information makes the proof in this paper much easier in comparison with \cite{Rhodes.1991, RS.2012}. Also we obtain an extra action 
we need to understand more fully.

Passing from elliptic maps to wreath products is not difficult: one just needs to number the edges leaving level $j$ with some index set $X_j$. 
The Lyndon--Chiswell construction does not yield uniform branching, but simply adds ``fake'' edges and leaves the function undefined on those. 
This leads to a wreath product of partial transformation semigroups.

Zeiger coding (see \cite{Zei.1967,Til.1976a,Til.1976b}) becomes the following. Suppose that $v_1,v_2$ are vertices of the tree at the 
same level such that $v_1r_{12} = v_2$ and $v_2r_{21} = v_1$, where $r_{12}$ and $r_{21}$ are elliptic maps representing 
elements of $\mathsf{KR}_{\operatorname{right}}(S,A)$. Then if real edges $E_1,\ldots,E_d$ descending from $v_1$ are labeled $1,\ldots,d$ by the index 
set, then the distinct $E_1r_{12},\ldots,E_dr_{12}$ edges get labeled $1,\ldots,d$ also by the index function.
As mentioned in the first two paragraphs, elliptic maps arise in studying the Zeiger proof of the Krohn--Rhodes Theorem and determining $\overline{S}$.

We plan to apply the results of this paper to Markov chains~\cite{RhodesSchilling.2019,ASST.2015,ASST.2015a}, where the tree is associated to a
statistic and strings running through the tree determine the stationary distribution and mixing time.

At the moment, it is very mysterious how the $\mathcal{J}$-order plays such a key role in both the semigroup theory and the Markov chain 
theory. In a future paper \cite{RSS}
we will compare the Chiswell--Lyndon trees of the Karnofsky--Rhodes and Rhodes expansions, leading to various
new statistics for the finite Markov chain.

This paper is organized as follows. In Section~\ref{section.KR}, we review the Karnofsky--Rhodes expansion of the
Cayley graph of a semigroup with a finite set of generators. In Section~\ref{section.CC}, we introduce the Dedekind
height function and our new Lyndon--Chiswell length function. The Lyndon--Chiswell length function is used in the 
Chiswell construction, which provides a rooted tree associated to the Karnofsky--Rhodes expansion of the semigroup.
The Chiswell construction in turn establishes the Holonomy Theorem (see Theorem~\ref{holt}). 
We conclude in Section~\ref{section.examples} with several examples.

%%%%%%%%%%%%%%%%%%%%%%%%%%%%%%%%%%%%%%%%%%%%%%%%%%%%%%%%%
\subsection*{Acknowledgments}
The authors are grateful to the anonymous referee for suggested corrections and improvements.

AS was partially supported by NSF grants DMS--1760329, DMS--1764153, and DMS--2053350. 
PVS was partially supported by CMUP, which is financed by national funds through FCT -- Funda\c c\~ao para a Ci\^encia e a Tecnologia, I.P., 
under the project with reference UIDB/00144/2020.

%%%%%%%%%%%%%%%%%%%%%%%%%%%%%%%%%%%%%%%%%%%%%%%%%%%%%%%%%%%%%%
\section{The Karnofsky--Rhodes expansion}
\label{section.KR}
%%%%%%%%%%%%%%%%%%%%%%%%%%%%%%%%%%%%%%%%%%%%%%%%%%%%%%%%%%%%%%

Let $A$ be a finite alphabet and let $A^+$ (respectively $A^*$) denote the free semigroup (respectively the free monoid) 
on $A$. If $\theta \colon A^+ \to S$ is a semigroup morphism onto a semigroup $S$, we say that $S$ is {\em generated by} 
$A$. We usually view $A$ as a subset of $S$. The reference to the morphism is omitted whenever possible and we use the 
notation $(S,A)$ to describe this situation.

We denote by $S^{\mathbbm{1}}$ the monoid obtained by adjoining to $S$ a (new) identity $\mathbbm{1}$ (even if 
$S$ is already a monoid). \defn{Green's quasi-orders} on $S$ are defined by
\begin{itemize}
\item
$a \leqslant_{{\mathcal{R}}} b \mbox{ if } a \in bS^{\mathbbm{1}}$,
\item
$a \leqslant_{{\mathcal{L}}} b \mbox{ if } a \in S^{\mathbbm{1}}b$,
\item
$a \leqslant_{{\mathcal{J}}} b \mbox{ if } a \in S^{\mathbbm{1}}bS^{\mathbbm{1}}$.
\end{itemize}
Then ${{\mathcal{X}}} = \, \leqslant_{{\mathcal{X}}} \cap \geqslant_{{\mathcal{X}}}$ for 
${\mathcal{X}} = {\mathcal{R}}, {\mathcal{L}}, {\mathcal{J}}$.

We denote by $S^{{\rm op}}$ the {\em opposite} semigroup of $S$, where the binary operation $\cdot$ on $S$ is replaced 
by the binary operation $x \circ y = y\cdot x$. Note that the $\mathcal{L}$ relation of $S^{{\rm op}}$ is the $\mathcal{R}$ 
relation of $S$, and the $\mathcal{R}$ relation of $S^{{\rm op}}$ is the $\mathcal{L}$ relation of $S$.

We define the \defn{left and right Cayley graphs} of $(S,A)$, denoted respectively by $\mathsf{LCay}(S,A)$ and 
$\mathsf{RCay}(S,A)$, as follows:
\begin{itemize}
\item
$S^{\mathbbm{1}}$ is the vertex set in both graphs,
\item
the edge set of $\mathsf{LCay}(S,A)$ is $\{ (s,a,as) \mid s \in S^{\mathbbm{1}},\, a \in A\}$,
\item
the edge set of $\mathsf{RCay}(S,A)$ is $\{ (s,a,sa) \mid s \in S^{\mathbbm{1}},\, a \in A\}$.
\end{itemize}
Note that these graphs are {\em complete} and {\em deterministic}: given a vertex $s$ and $u \in A^+$, there exists a 
unique path with label $u$ starting at $s$. The following remark, which follows from the definitions, will allow us to use 
left-right symmetries:

\begin{remark}
\label{lrop}
$\mathsf{LCay}(S,A) = \mathsf{RCay}(S^{{\rm op}},A)$.
\end{remark}

An edge $(p,q)$ of a directed graph is called a \defn{transition edge} if there exists no path from $q$ to $p$. This applies 
also to $A$-labeled graphs (in particular to left and right Cayley graphs), where $(s,a,s')$ is a transition edge if there is 
no path from $s'$ to $s$. Note that in a Cayley graph, edges of the form $(\mathbbm{1},a,a)$ are always transition edges.

If $s \xrightarrow{a} s'$ is an edge of $\mathsf{RCay}(S,A)$, then $s' = sa$ and so $s' \leqslant_{\mathcal{R}} s$. Hence 
this edge is a transition edge if and only if $s' <_{\mathcal{R}} s$. Note also that if two transition edges occur in two 
different paths, they must occur {\em in the same order}.

The \defn{right Karnofsky--Rhodes expansion} $\mathsf{KR}_{\operatorname{right}}(S,A)$ of $(S,A)$ is defined as the 
quotient $A^+/\tau_{r}$, where $\tau_{r}$ is the congruence on $A^+$ defined as follows: $u\,\tau_{r}\, v$ if $u = v$ holds 
in $S$ and the paths $\mathbbm{1} \xrightarrow{u} u$ and $\mathbbm{1} \xrightarrow{v} v$ in $\mathsf{RCay}(S,A)$ have 
the same transition edges. Then $S$ is a homomorphic image of $\mathsf{KR}_{\operatorname{right}}(S,A)$ in the 
obvious way.

The \defn{left Karnofsky--Rhodes expansion} of $(S,A)$ can be defined by
$$\mathsf{KR}_{\operatorname{left}}(S,A) = \mathsf{KR}_{\operatorname{right}}(S^{{\rm op}},A).$$

We will be paying particular attention to $\mathsf{KR}^{\mathbbm{1}}_{\operatorname{right}}(S,A)$, which is obtained 
by adjoining the (new) identity $\mathbbm{1}$ to $\mathsf{KR}_{\operatorname{right}}(S,A)$. We can view 
$\mathsf{KR}^{\mathbbm{1}}_{\operatorname{right}}(S,A)$ as the quotient $A^*/(\tau_r \cup \{ (1,1)\})$. 
Similarly, we define $\mathsf{KR}^{\mathbbm{1}}_{\operatorname{left}}(S,A)$.

%%%%%%%%%%%%%%%%%%%%%%%%%%%%%%%%%%%%%%%%%%%%%%%%%%%%%%%%%%%%%%
\section{The Chiswell construction}
\label{section.CC}
%%%%%%%%%%%%%%%%%%%%%%%%%%%%%%%%%%%%%%%%%%%%%%%%%%%%%%%%%%%%%%

From now on, $S$ is a fixed finite semigroup and $A$ is a generating set of $S$.

%%%%%%%%%%%%%%%%%%%%%%%%%%%%%%%%%%%%%%%%%%%%%%%%%%%%%%%%%%%%%%
\subsection{The Dedekind height function}

We shall write $T = \mathsf{KR}_{\operatorname{right}}(S,A)$ throughout this section, and let 
$\varphi \colon T^{\mathbbm{1}} \to S^{\mathbbm{1}}$ denote the canonical surmorphism.

The \defn{Dedekind height function} $h:S^{\mathbbm{1}} \to \mathbb{N}$ is defined as
\[
	h(s) = \max\{ k \in \mathbb{N} \mid \mbox{there exists a chain $s_0 >_{\mathcal{J}} \cdots >_{\mathcal{J}} s_k = s$ in }
	S^{\mathbbm{1}}\}.
\]
This should be denoted $h_S$, but the semigroup $S$ is usually
understood, as in Proposition \ref{dhf} below.

Finite semigroups are known to be \defn{stable}: they satisfy the equalities
\[
	\leqslant_{\mathcal{R}} \cap\, {\mathcal{J}} = {\mathcal{R}}, \quad \leqslant_{\mathcal{L}} \cap\, {\mathcal{J}} 
	= {\mathcal{L}}.
\]
The following result will prove useful later.

\begin{lemma}
\label{stab}
If $s <_{\mathcal{K}} s'$ holds in $S$ for ${\mathcal{K}} \in \{ {\mathcal{R}}, {\mathcal{L}}, {\mathcal{J}}\}$, then $h(s) > h(s')$.
\end{lemma}

\begin{proof}
The result is immediate for ${\mathcal{J}}$. By symmetry, we may assume that $s <_{\mathcal{R}} s'$. It follows that 
$s \leqslant_{\mathcal{J}} s'$. Now since $S$ is stable we cannot have $s {\mathcal{J}} s'$, thus
$s <_{\mathcal{J}} s'$ and so $h(s) > h(s')$.
\end{proof}

A semigroup $S$ is \defn{regular} if every $s\in S$ is regular. That is, for each $s\in S$ there exists an element $s' \in S$ such 
that $ss's=s$.

\begin{lemma}
\label{pullreg}
If $t,t' \in T^{\mathbbm{1}}$ satisfy $\varphi(tt't) = \varphi(t)$, then $tt't = t$.
\end{lemma}

\begin{proof}
Let $u,v \in A^*$ represent $t$ and $t'$, respectively. We have paths $\mathbbm{1} \xrightarrow{u} \varphi(t)$ and 
$\mathbbm{1} \xrightarrow{uvu} \varphi(tt't)$ in $\mathsf{RCay}(S,A)$. Since $\varphi(tt't) = \varphi(t)$ and 
$\mathsf{RCay}(S,A)$ is deterministic, we actually have a loop labeled by $vu$ at $\varphi(t)$. Since a loop cannot 
contain transition edges, it follows that $uvu \,\tau_r\, u$ and so $tt't = t$.
\end{proof}

It follows that if $S$ is regular, then $T^{\mathbbm{1}}$ is also regular.

\begin{lemma}
\label{krpj}
Assume that $S$ is a regular semigroup and let $t,t' \in T^{\mathbbm{1}}$. 
\begin{itemize}
\item[(i)]
If  $\varphi(t) \leqslant_{\mathcal{J}} \varphi(t')$ if and only if $t \leqslant_{\mathcal{J}} t'$.
\item[(ii)]
$\varphi(t) <_{\mathcal{J}} \varphi(t')$ if and only if $t <_{\mathcal{J}} t'$.
\end{itemize}
\end{lemma}

\begin{proof}
(i) If $\varphi(t) \leqslant_{\mathcal{J}} \varphi(t')$, there exist $p,q \in T^{\mathbbm{1}}$ such that $\varphi(t) = \varphi(pt'q)$. On the other hand, 
since $S$ is regular, we have $\varphi(t) = \varphi(tzt)$ for some $z \in T$. Hence 
$$\varphi(t) = \varphi(tzt) = \varphi(tztzt) = \varphi(tzpt'qzt)$$
and it follows from Lemma \ref{pullreg} that $t = tzpt'qzt$. Therefore $t \leqslant_{\mathcal{J}} t'$.

The converse implication follows from $\varphi$ being a homomorphism.

(ii) Assume that $\varphi(t) <_{\mathcal{J}} \varphi(t')$. By (i), we obtain $t \leqslant_{\mathcal{J}} t'$. Since $\leqslant_{\mathcal{J}}$ is 
preserved by homomorphisms, 
$t\, \mathcal{J} \, t'$ implies $\varphi(t) \, \mathcal{J} \, \varphi(t')$, a contradiction. Thus $t <_{\mathcal{J}} t'$.

Conversely, assume that $t <_{\mathcal{J}} t'$. Hence $\varphi(t) \leqslant_{\mathcal{J}} \varphi(t')$. Since $\varphi(t) \, \mathcal{J} \, \varphi(t')$ implies 
$t\, \mathcal{J} \, t'$ by (i), we get $\varphi(t) <_{\mathcal{J}} \varphi(t')$.
\end{proof}

\begin{proposition}
\label{dhf}
Assume that $S$ is a regular semigroup and let $t \in T^{\mathbbm{1}}$. Then $h(t) = h(\varphi(t))$. 
\end{proposition}

\begin{proof}
By Lemma \ref{krpj}(ii), we have a chain
\[
	t_1 >_{\mathcal{J}} \cdots >_{\mathcal{J}} t_k = t
\]
in $T$ if and only if we have a chain
\[
	\varphi(t_1) >_{\mathcal{J}} \cdots >_{\mathcal{J}} \varphi(t_k) = \varphi(t)
\]
in $S$. Thus $h(t) = h(\varphi(t))$.
\end{proof}

\begin{remark}
If $(S,A)$ is not regular, computing $h(t)$ for $t\in T$ can be more challenging sometimes.
\end{remark}

%%%%%%%%%%%%%%%%%%%%%%%%%%%%%%%%%%%%%%%%%%%%%%%%%%%%%%%%%%%%%%
\subsection{The Lyndon--Chiswell length function}

Write
\[
	\ell = 2 \max\{ h(s) \mid s \in S^{\mathbbm{1}} 
	\}.
\]
Denote by $t(E)$ the endpoint of an edge $E$ of a directed graph.

Let $\alpha,\beta \in T^{\mathbbm{1}}$. 
Let $(E_1,\ldots,E_m)$ and $(E'_1,\ldots,E'_n)$ be the corresponding sequences of transition edges.  Since any edge 
starting at $\mathbbm{1}$ is a transition edge, we have $m = 0$ if and only if $\alpha = \mathbbm{1}$.
Let 
$$\xi(\alpha,\beta) = 
\max\{ i \in \{ 0,\ldots,m\} \mid E_1 = E'_1,\ldots, E_i = E'_i \}.$$
Hence $\xi(\alpha,\beta)$ counts the maximum number of transition edges consecutively shared by $\alpha$ and $\beta$, 
when we start with the first and proceed in order. If $\xi(\alpha,\beta) = k > 0$, we write also $\eta(\alpha,\beta) = E_k = E'_k$.

\begin{lemma}
\label{trin}
For all $\alpha,\beta,\gamma \in T^{\mathbbm{1}}$, we have:
\begin{itemize}
\item[(i)]
$\xi(\alpha,\beta) = \xi(\beta,\alpha)$;
\item[(ii)]
$\xi(\alpha\gamma,\beta\gamma) \geqslant \xi(\alpha,\beta)$;
\item[(iii)]
$\xi(\alpha,\gamma) \geqslant \min(\xi(\alpha,\beta),\xi(\beta,\gamma))$.
\end{itemize}
\end{lemma}

\begin{proof}
(i) follows from the symmetry of equality.

(ii) follows from the following fact: the sequence of transition edges of $\alpha\gamma$ starts with the sequence of 
transition edges of $\alpha$.

For (iii), we may assume that $\min(\xi(\alpha,\beta),\xi(\beta,\gamma)) = k > 0$. Let $(E_1,\ldots,E_m)$, $(E'_1,\ldots,E'_n)$ and $(E''_1,\ldots,E''_p)$ be the sequences of transition edges corresponding to $\alpha$, $\beta$ and $\gamma$. Then $E_1 = E'_1,\ldots, E_k = E'_k$  and also $E'_1 = E''_1,\ldots, E'_k = E''_k$.  Hence $E_1 = E''_1,\ldots, E_k = E''_k$  and so $\xi(\alpha,\gamma) \geqslant k$ as required.
\end{proof}

We prove also the following result:

\begin{lemma}
\label{leta}
Let $\alpha,\beta,\gamma \in T^{\mathbbm{1}}$ be such that $\xi(\alpha,\beta) > 0$. Then:
\begin{itemize}
\item[(i)]
$\xi(\gamma\alpha,\gamma\beta) > 0$;
\item[(ii)]
$t(\eta(\gamma\alpha,\gamma\beta)) \leqslant_{{\mathcal{R}}} \varphi(\gamma) t(\eta(\alpha,\beta)) \leqslant_{{\mathcal{L}}} t(\eta(\alpha,\beta))$.
\end{itemize}
\end{lemma}

\begin{proof}
We may assume that $\gamma \neq \mathbbm{1}$. 

(i) Since the first letter of a word representing $\gamma$ must necessarily label a transition edge of $\gamma$ (or $\gamma\alpha$, or $\gamma\beta$), it follows that $\xi(\gamma\alpha,\gamma\beta) > 0$.

(ii) Let $(E_1,\ldots,E_m)$, $(E'_1,\ldots,E'_n)$ and $(E''_1,\ldots,E''_p)$ be the sequences of transition edges corresponding to $\alpha$, $\beta$ and $\gamma$, respectively. Let $\xi(\alpha,\beta) = k$, so that $\eta(\alpha,\beta) = E_k = E'_k$. What are the possible transition edges of $\gamma\alpha$? Clearly, $E''_1,\ldots,E''_p$ are all transition edges of $\gamma\alpha$.

Write $S^{\mathbbm{1}} = A^*/\sigma$ and $T^{\mathbbm{1}} = A^*/\tau$. Let $u = e_1u_1\ldots e_mu_m$ and $u' = e'_1u'_1\ldots e'_nu'_n$ be words representing $\alpha$ and $\beta$ respectively, 
where $e_i$ and $e'_j$ denote the labels of $E_i$ and $E'_j$. Write $v_i = e_1u_1\ldots e_iu_i$ and 
$v'_j = e'_1u'_1\ldots e'_ju'_j$ for all $0 \leqslant i \leqslant m$ and $0 \leqslant j \leqslant n$. Let $\gamma = w\tau$. 

Since the letters occurring in the $u_i$ label no transition edges in the path
$\mathbbm{1} \xrightarrow{u} u\sigma$ in $\mathsf{RCay}(S,A)$, there exists a path $v_i\sigma \xrightarrow{x_i} (v_{i-1}e_i)\sigma$ in $\mathsf{RCay}(S,A)$ for $i = 1,\ldots,m$. Hence $(v_ix_i)\tau = (v_{i-1}e_i)\tau$ and so $(wv_ix_i)\tau = (wv_{i-1}e_i)\tau$ for $i = 1,\ldots,m$. Thus the only possible transition edges of $\gamma\alpha$ beyond $E''_1,\ldots,E''_p$ are of the form $(wv_{i-1})\sigma \xrightarrow{e_i} (wv_{i-1}e_i)\sigma$ for some $i \in \{1,\ldots,m\}$. Similarly, the only possible transition edges of $\gamma\beta$ beyond $E''_1,\ldots,E''_p$ are of the form $(wv'_{i-1})\sigma \xrightarrow{e'_i} (wv'_{i-1}e'_i)\sigma$ for some $i \in \{1,\ldots,n\}$.

Let 
$$I = \{ i \in \{ 1,\ldots,k\} \mid (wv_{i-1})\sigma \xrightarrow{e_i} (wv_{i-1}e_i)\sigma \mbox{ is a transition edge of  }\mathsf{RCay}(S,A)\}.$$
Note that $I$ needs not to contain all the integers between 1 and $r$.
Suppose first that $I \neq \emptyset$ and let $r = \max I$. Since 
$$(v_{i-1}\sigma,e_i,(v_{i-1}e_i)\sigma) = E_i = E'_i = (v'_{i-1}\sigma,e'_i,(v'_{i-1}e'_i)\sigma)$$
for $i = 1,\ldots,k$, we get
$((wv_{i-1})\sigma,e_i,(wv_{i-1}e_i)\sigma) = ((wv'_{i-1})\sigma,e'_i,(wv'_{i-1}e'_i)\sigma)$
as well. Hence $\gamma\alpha$ and $\gamma\beta$ share the same transition edges up to 
$(wv_{r-1}\sigma,e_r,(wv_{r-1}e_r)\sigma) = (wv'_{r-1}\sigma,e'_r,(wv'_{r-1}e'_r)\sigma)$ at least. Now since $r = \max I$ there are no more transition edges between $(wv_{r-1}e_r)\sigma$ and $(wv_{k-1}e_k)\sigma$ in $\mathsf{RCay}(S,A)$. Hence
$$t(\eta(\gamma\alpha,\gamma\beta)) \leqslant_{{\mathcal{R}}} (wv_{r-1}e_r)\sigma \, {\mathcal{R}}\, (wv_{k-1}e_k)\sigma =
\varphi(\gamma) t(\eta(\alpha,\beta)).$$

We reach the same conclusion in the case $I = \emptyset$, replacing the edge $E_r$ in the above argument by $E''_p$ (note that $p \geq 1$ since we are assuming $\gamma \neq \mathbbm{1}$).

Finally, $\gamma t(\eta(\alpha,\beta)) \leqslant_{{\mathcal{L}}} t(\eta(\alpha,\beta))$ holds trivially.
\end{proof}

Let $(E_1,\ldots,E_m)$ and $(E'_1,\ldots,E'_n)$ be the sequences of transition edges corresponding to $\alpha$ and $\beta$, respectively. Let $k = \xi(\alpha,\beta)$. If $k > 0$, we have $\eta(\alpha,\beta) = E_k = E'_k$. We 
define the \defn{Lyndon--Chiswell length function} 
$D \colon T^{\mathbbm{1}} \times T^{\mathbbm{1}} \to \mathbb{N}$ by
\[
	D(\alpha,\beta) = \begin{cases}
	\ell& \text{if $\alpha = \beta$,} \\
	0 & \text{if $\alpha \neq \beta$ and $k = 0$,} \\
	2h(t(E_{k})) & \text{if $0 < k < m,n$ 
	and $t(E_{k+1}) = t(E'_{k+1})$,} \\
	2h(t(E_{k}))-1 & \text{in all remaining cases.}	
	\end{cases}
\]
Note that ${\rm im}(h) = \{ 0,1,\ldots, \frac{\ell}{2} \}$ implies ${\rm im}(D) \subseteq \{0,1,\ldots, \ell\}$.
We show now that
\begin{equation}
\label{oop}
(\alpha \neq \beta \; \wedge \; \xi(\alpha,\beta) < \xi(\alpha,\gamma))\; \Rightarrow \; D(\alpha,\beta) < D(\alpha,\gamma)
\end{equation}
holds for all $\alpha,\beta,\gamma \in T^{\mathbbm{1}}$.

Assume that $\xi(\alpha,\beta) < \xi(\alpha,\gamma)$. We may assume that $\xi(\alpha,\beta) > 0$, otherwise 
$D(\alpha,\beta) = 0$. Then $t(\eta(\alpha,\beta)) >_{\mathcal{R}} t(\eta(\alpha,\gamma))$ 
because there exists in $\mathsf{RCay}(S,A)$ a path from 
$t(\eta(\alpha,\beta))$ to $t(\eta(\alpha,\gamma))$ containing transition edges. By Lemma \ref{stab}, 
we get $h(t(\eta(\alpha,\beta))) < h(t(\eta(\alpha,\gamma)))$, yielding 
$D(\alpha,\beta) < D(\alpha,\gamma)$. Thus~\eqref{oop} holds.

The following properties go a little beyond those of \cite[Fact 1.9]{Rhodes.1991}.
We provide a full proof.

\begin{lemma}
\label{lemma.length property}
The Lyndon-Chiswell length function satisfies the following properties for all $\alpha,\beta, \gamma \in T^{\mathbbm{1}}$:
\begin{itemize}
\item[(i)]
$D(\alpha,\beta)=D(\beta,\alpha)$;
\item[(ii)]
$D(\alpha\gamma, \beta\gamma) \geqslant D(\alpha,\beta)$;
\item[(iii)]
$D(\gamma\alpha, \gamma\beta) \geqslant D(\alpha,\beta)$;
\item[(iv)]
(\textbf{isoperimetric inequality}) $D(\alpha,\gamma) \geqslant \min(D(\alpha,\beta),D(\beta,\gamma))$.
\end{itemize}
\end{lemma}

\begin{proof}
(i) It follows easily from Lemma \ref{trin}(i). 

(ii) We may assume that $\alpha\gamma \neq \beta\gamma$, otherwise $D(\alpha\gamma, \beta\gamma) = \ell$ is maximum. Hence $\alpha \neq \beta$ as well. We may also assume that $\xi(\alpha,\beta) > 0$, otherwise $D(\alpha,\beta) = 0$.

Let $(E_1,\ldots,E_m)$ and $(E'_1,\ldots,E'_n)$ be the sequences of transition edges corresponding to $\alpha$ and $\beta$. Let $k =  \xi(\alpha,\beta) > 0$, so that $E_1 = E'_1,\ldots, E_k = E'_k$. 
By the proof of Lemma \ref{trin}(ii), the sequences of transition edges corresponding to $\alpha\gamma$ and $\beta\gamma$ are of the form $(E_1,\ldots,E_m,F_1,\ldots,F_r)$ and $(E'_1,\ldots,E'_n,F'_1,\ldots,F'_s)$. Suppose that $\xi(\alpha\gamma,\beta\gamma) > \xi(\alpha,\beta) = k$. 
Then either $m = k < n$ and $F_1 = E'_{k+1}$, or $n = k < m$ and $E_{k+1} = F'_1$, or $m = n = k$ and $F_{1} = F'_1$.
In any case, we have $t(\eta(\alpha\gamma,\beta\gamma)) <_{{\mathcal{R}}} t(\eta(\alpha,\beta))$. Now
it follows from Lemma \ref{stab} that $h(t(\eta(\alpha\gamma,\beta\gamma))) > h(t(\eta(\alpha,\beta)))$. Thus $D(\alpha\gamma, \beta\gamma) > D(\alpha,\beta)$.

Therefore we may assume by Lemma \ref{trin}(ii) that $\xi(\alpha\gamma,\beta\gamma) = \xi(\alpha,\beta) = k$, which means that $D(\alpha\gamma, \beta\gamma)$ and $D(\alpha,\beta)$ differ by at most 1.
Hence we may also assume that $t(E_{k+1}) = t(E'_{k+1})$ (the case where we do not subtract 1). Since $E_1,\ldots,E_{k+1}$ and $E'_1,\ldots,E'_{k+1}$ are the first $k+1$ transition edges corresponding to $\alpha\gamma$ and $\beta\gamma$, respectively, we immediately get $D(\alpha\gamma, \beta\gamma) = D(\alpha,\beta)$.

(iii) Similarly to the proof of (ii), we may assume that $\gamma\alpha \neq \gamma\beta$ and $\xi(\alpha,\beta) > 0$.
By Lemma~\ref{leta}(ii), we have $t(\eta(\gamma\alpha,\gamma\beta)) \leqslant_{{\mathcal{J}}} t(\eta(\alpha,\beta))$.

Suppose first that $t(\eta(\gamma\alpha,\gamma\beta)) <_{{\mathcal{J}}} t(\eta(\alpha,\beta))$. Then Lemma \ref{stab} yields 
$h(t(\eta(\gamma\alpha,\gamma\beta))) > h(t(\eta(\alpha,\beta)))$
and so
$$D(\gamma\alpha, \gamma\beta) \geqslant 2h(t(\eta(\gamma\alpha,\gamma\beta)))-1 > 2h(t(\eta(\alpha,\beta)))
\geqslant D(\alpha,\beta).$$

Thus we may assume that $t(\eta(\gamma\alpha,\gamma\beta)) \, {\mathcal{J}}\, t(\eta(\alpha,\beta))$, so that 
$h(t(\eta(\gamma\alpha,\gamma\beta))) = h(t(\eta(\alpha,\beta)))$. We may also assume that $D(\alpha,\beta)$ is even, so 
that $k < m,n$ and $t(E_{k+1}) = t(E'_{k+1})$. 

Recall from the proof of Lemma \ref{leta} the words
$u = e_1u_1\ldots e_mu_m$ and $u' = e'_1u'_1\ldots e'_nu'_n$ representing $\alpha$ and $\beta$, and all the associated notation. By Lemma \ref{leta}(ii), we get 
$$t(\eta(\gamma\alpha,\gamma\beta)) \leqslant_{{\mathcal{R}}} \varphi(\gamma) t(\eta(\alpha,\beta)) = (wv_{k-1}e_{k})\sigma.$$
Thus
$$(wv_ke_{k+1})\sigma \leqslant_{{\mathcal{L}}}  (v_ke_{k+1})\sigma <_{{\mathcal{R}}} (v_{k-1}e_{k})\sigma = t(\eta(\alpha,\beta)) \, {\mathcal{J}}\, t(\eta(\gamma\alpha,\gamma\beta)) \leqslant_{{\mathcal{R}}} 
(wv_{k-1}e_{k})\sigma \, {\mathcal{R}}\, (wv_k)\sigma.$$
Since finite semigroups are stable, the relation $<_{{\mathcal{R}}}$ is contained in $<_{{\mathcal{J}}}$, hence 
$(wv_ke_{k+1})\sigma <_{{\mathcal{J}}} (wv_k)\sigma$. Since $(wv_ke_{k+1})\sigma \leqslant_{{\mathcal{R}}} (wv_{k})\sigma$, it follows that $(wv_ke_{k+1})\sigma <_{{\mathcal{R}}} (wv_k)\sigma$. Hence 
$$(wv_k)\sigma \xrightarrow{e_{k+1}} (wv_ke_{k+1})\sigma$$
is a transition edge, in fact the first transition edge of $\gamma\alpha$ which is not shared with $\gamma\beta$. Similarly,
$$(wv'_k)\sigma \xrightarrow{e'_{k+1}} (wv'_ke'_{k+1})\sigma$$
is the first transition edge of $\gamma\beta$ which is not shared with $\gamma\alpha$. Now
$$(v_ke_{k+1})\sigma = t(E_{k+1}) = t(E'_{k+1}) = (v'_ke'_{k+1})\sigma$$
yields 
$(wv_ke_{k+1})\sigma =  (wv'_ke'_{k+1})\sigma$ and so
$$D(\gamma\alpha, \gamma\beta) = 2h(t(\eta(\gamma\alpha,\gamma\beta))) = 2h(t(\eta(\alpha,\beta)))
= D(\alpha,\beta).$$
Therefore (iii) holds.

(iv) We may assume that $\alpha,\beta,\gamma$ are all different. Let $(E''_1,\ldots,E''_p)$ be the sequence of transition edges corresponding to $\gamma$.

If $\xi(\alpha,\beta) > \xi(\beta,\gamma)$, we can exchange $\alpha$ and $\gamma$ in view of (i), and so Lemma \ref{trin}(i) allows us to assume that $\xi(\alpha,\beta) \leqslant \xi(\beta,\gamma)$. Now Lemma \ref{trin}(iii) yields
$\xi(\alpha,\gamma) \geqslant  \xi(\alpha,\beta)$. Since $\xi(\alpha,\gamma) > \xi(\alpha,\beta)$ immediately implies our objective in view of (\ref{oop}), we are now restricted to the case
$$0 < k = \xi(\alpha,\gamma) = \xi(\alpha,\beta) \leqslant \xi(\beta,\gamma).$$
It follows that $E_i = E'_i = E''_i$ for $i = 1,\ldots,k$.

Unless $t(E_{k+1}) = t(E'_{k+1})$, we get $D(\alpha,\beta) = 2h(t(E_{k}))-1$ and we are done. Hence we may assume that $t(E_{k+1}) = t(E'_{k+1})$. Now, unless $t(E'_{k+1}) = t(E''_{k+1})$, we get $D(\beta,\gamma) = 2h(t(E_{k}))-1$ and we are done as well. But then $t(E_{k+1}) = t(E''_{k+1})$ and so $D(\alpha,\gamma) = 2h(t(E_{k})) = D(\alpha,\beta)$.
Therefore (iv) holds.
\end{proof}

\begin{remark}
In view of these properties, $D$ can indeed be called a length function for (unexpectedly) both a left and right action because of Lemma~\ref{lemma.length property} (ii) and (iii).
\end{remark}

%%%%%%%%%%%%%%%%%%%%%%%%%%%%%%%%%%%%%%%%%%%%%%%%%%%%%%%%
\subsection{Representations as elliptic maps on a rooted tree}

Let $\Gamma = (V,E)$ be a simple undirected graph. Then $\Gamma$ is a tree if it is connected and admits no cycles (i.e. no closed paths of 
the form $v_1 \edge \cdots \edge v_n \edge v_1$ with $n \geq 3$ different vertices). 
If we distinguish a vertex $v_0 \in V$, we get the \defn{rooted tree} $(\Gamma,v_0)$.

Given a rooted tree $(\Gamma,v_0)$, we get a \defn{depth function} $\delta:V \to \mathbb{N}$ as follows: $\delta(v)$ 
is the edge length of the shortest path connecting $v$ to $v_0$. An \defn{endomorphism} of the rooted tree $(\Gamma,v_0)$ 
is a function $\varphi:V \to V$ such that:
\begin{itemize}
\item
$\delta(\varphi(v)) = \delta(v)$ for every $v \in V$;
\item
if $v \edge w$ is an edge of $\Gamma$, so is $\varphi(v) \edge \varphi(w)$.
\end{itemize}
Endomorphisms of rooted trees are also known as \defn{elliptic maps}. We denote by EM$(\Gamma,v_0)$ the 
monoid of all elliptic maps of $(\Gamma,v_0)$.

A \defn{representation} of a monoid $M$ as elliptic maps on a rooted tree $(\Gamma,v_0)$ is a monoid homomorphism 
$\theta:M \to {\rm EM}(\Gamma,v_0)$. The representation is \defn{faithful} if $\varphi$ is one-to-one.

%%%%%%%%%%%%%%%%%%%%%%%%%%%%%%%%%%%%%%%%%%%%%%%%%%%%%%%%
\subsection{The Chiswell construction and the holonomy theorem}

We adapt next the Chiswell construction described in~\cite[Proof of Theorem 1.12]{Rhodes.1991} 
and~\cite[Proof of Theorem 4.7]{RS.2012} (see also~\cite{Chiswell.1976}). 

Let $T = \mathsf{KR}_{\operatorname{right}}(S,A)$ and let $D \colon T^{\mathbbm{1}} \times T^{\mathbbm{1}} \to \mathbb{N}$
be the Lyndon--Chiswell length function defined before (with %image $\{ 0,1,\ldots,\ell\}$). 
maximum value $\ell$).
Write
$$C = \{(k,\alpha) \mid 0\leqslant k\leqslant \ell, \alpha \in T^{\mathbbm{1}}\}.$$
We define a relation~$\sim$ on $C$ by $(k,\alpha) \sim (k',\beta)$ if:
\begin{itemize}
\item
$k=k'$;
\item
$D(\alpha,\beta) \geqslant k$. 
\end{itemize}
It follows from Lemma \ref{lemma.length property}(i) and (iv) that $\sim$ is indeed an
equivalence relation on $C$. Note that $(0,\alpha) \sim (0,\beta)$ for all $\alpha,\beta \in T^{\mathbbm{1}}$.

Denote by $[k,\alpha]$ the equivalence class of $(k,\alpha) \in C$. Define an undirected graph $\mathcal{C}$ with vertices 
$[k,\alpha]$ and edges $[k,\alpha] \edge [k+1,\alpha]$ when $0\leqslant k < \ell$ and $\alpha \in T^{\mathbbm{1}}$.
Note that 
\begin{equation}
\label{tree1}
\mbox{if $[k,\beta] \edge [k+1,\alpha]$ is an edge of $\mathcal{C}$ then $[k,\beta] = [k,\alpha]$.}
\end{equation}
Indeed, if there exists such an edge then there exists some $\gamma \in T^{\mathbbm{1}}$ such that $[k,\beta] = [k,\gamma]$
and $[k+1,\alpha] = [k+1,\gamma]$. 
It follows that
$D(\alpha,\gamma) \geqslant k+1 > k$. Hence $[k,\alpha] = [k,\gamma]$ and (\ref{tree1}) holds.

With minimal adaptations from ~\cite{Rhodes.1991} and~\cite{RS.2012}, we prove the following lemma for the sake of 
completeness.

\begin{lemma}
\label{tree}
$(\mathcal{C},[0,\mathbbm{1}])$ is a rooted tree.
\end{lemma}

\begin{proof}
We have a path
$$[0,\mathbbm{1}] = [0,\alpha] \edge [1,\alpha] \edge \cdots \edge [k,\alpha]$$
for every vertex $[k,\alpha]$, hence $\mathcal{C}$ is connected.

Suppose that
$$[k_0,\alpha_0] \edge [k_1,\alpha_1] \edge \cdots \edge [k_n,\alpha_n] = [k_0,\alpha_0]$$
is a cycle in $\mathcal{C}$. We may assume that $k_0 \geqslant k_i$ for every $0 \leqslant i \leqslant n$. Then $k_1 = k_{n-1} = k_0-1$ 
and it follows from (\ref{tree1}) that $[k_1,\alpha_1] = [k_0-1,\alpha_0] = [k_{n-1},\alpha_{n-1}]$, a contradiction. Therefore $\mathcal{C}$ 
is a tree as required.
\end{proof}

This rooted tree is the \defn{Chiswell tree} induced by the Lyndon--Chiswell length function $D \colon T^{\mathbbm{1}} \times T^{\mathbbm{1}} \to \mathbb{N}$.
Note that the depth function is given by $\delta([k,\alpha]) = k$.

\begin{theorem}[\defn{Holonomy Theorem}]
\label{holt}
Let $(S,A)$ be a finite semigroup $S$ with generators $A$. Then
$\mathsf{KR}^{\mathbbm{1}}_{\operatorname{right}}(S,A)$ and $\mathsf{KR}^{\mathbbm{1}}_{\operatorname{left}}(S,A)$ are faithfully represented 
as elliptic maps on a finite rooted tree.
\end{theorem}

\begin{proof}
Once again, we adapt the proof from~\cite{Rhodes.1991,RS.2012}.

Let 
$$\begin{array}{rcl}
\epsilon:\mathsf{KR}_{\operatorname{right}}(S,A)&\to&{\rm EM}((\mathcal{C},[0,\mathbbm{1}]))\\
\alpha&\mapsto&\epsilon_{\alpha}
\end{array}$$
be defined by 
$$\epsilon_{\alpha}([k,\beta]) = [k,\alpha\beta].$$
First, we show that $\epsilon_{\alpha}$ is well defined. Suppose that $[k,\beta] = [k',\beta']$. Then $k = k'$ and $D(\beta,\beta') \geqslant k$. Then $D(\alpha\beta,\alpha\beta') \geqslant k$ by 
Lemma \ref{lemma.length property}(iii) and so $[k,\alpha\beta] = [k',\alpha\beta']$. Thus $\epsilon_{\alpha}$ is well defined.

It is obvious that $\delta(\epsilon_{\alpha}([k,\beta])) = k = \delta([k,\beta])$.  On the other hand, if $[k,\beta] \edge [k+1,\beta]$ 
is an edge of $\mathcal{C}$, so is $\epsilon_{\alpha}([k,\beta]) = [k,\alpha\beta] \edge [k+1,\alpha\beta] 
= \epsilon_{\alpha}([k+1,\beta])$. Therefore $\epsilon_{\alpha}$ is an elliptic map on the finite rooted tree 
$(\mathcal{C},[0,\mathbbm{1}])$ and so $\epsilon$ is well defined.

Given $\alpha,\alpha' \in \mathsf{KR}^{\mathbbm{1}}_{\operatorname{right}}(S,A)$, we have
$$\epsilon_{\alpha\alpha'}([k,\beta]) = [k,\alpha\alpha'\beta] = \epsilon_{\alpha}(\epsilon_{\alpha'}([k,\beta])),$$
hence $\epsilon_{\alpha\alpha'} = \epsilon_{\alpha}\epsilon_{\alpha'}$. On the other hand,
$\epsilon_{\mathbbm{1}}([k,\beta]) = [k,\beta]$ and so $\epsilon_{\mathbbm{1}}$ is the identity map. Thus $\epsilon$ 
is a monoid homomorphism.

Finally, assume that $\epsilon_{\alpha} = \epsilon_{\alpha'}$. Then in particular
$$[\ell,\alpha] = \epsilon_{\alpha}([\ell,\mathbbm{1}]) = \epsilon_{\alpha'}([\ell,\mathbbm{1}]) = [\ell,\alpha'],$$
hence $D(\alpha,\alpha') \geqslant \ell = \max({\rm im}(D))$. 

Suppose that $\alpha \neq \alpha'$. Let $(E_1,\ldots,E_m)$ be the sequence of transition edges corresponding to $\alpha$.
Since
$$D(\alpha,\alpha') = \ell = 2 \max\{ h(s) \mid s \in S^{\mathbbm{1}} \},$$
we are not subtracting 1, which implies that $\alpha$ possesses transition edges beyond $E_{\xi(\alpha,\alpha')}$, i.e. $\xi(\alpha,\alpha') < m$. In view of Lemma \ref{stab}, this contradicts the fact that $h$ should reach its maximum value at $t(E_{\xi(\alpha,\alpha')}))$.
Thus 
$\alpha = \alpha'$ and so $\epsilon$ is one-to-one. Therefore the representation is faithful.

Recall now that $\mathsf{KR}_{\operatorname{left}}(S,A) = \mathsf{KR}_{\operatorname{right}}(S^{{\rm op}},A).$ It follows 
from the first part that $\mathsf{KR}^{\mathbbm{1}}_{\operatorname{left}}(S,A)$ is also faithfully represented as elliptic maps 
on a finite rooted tree.
\end{proof}

\begin{remark}
It follows easily from Lemma \ref{lemma.length property}(ii) that we can consider a right action of $\mathsf{KR}_{\operatorname{right}}(S,A)$ 
on the Chiswell tree $(\mathcal{C},[0,\mathbbm{1}])$ given by
$$[k,\beta]\alpha = [k,\beta\alpha].$$
A straightforward adaptation of the proof of Theorem \ref{holt} shows that
we obtain an injective monoid homomorphism
$$\epsilon':\mathsf{KR}_{\operatorname{right}}(S,A) \to ({\rm EM}((\mathcal{C},[0,\mathbbm{1}])))^{{\rm op}},$$
since here the elliptic mappings must compose from left to right.
\end{remark}

\begin{remark}
Note that the Lyndon-Chiswell length function on $\mathsf{KR}^{\mathbbm{1}}_{\operatorname{right}}(S^{{\rm op}},A) \times 
\mathsf{KR}^{\mathbbm{1}}_{\operatorname{right}}(S^{{\rm op}},A)$ would be the version of the Lyndon-Chiswell length 
function built from $S$ when we replace its right Cayley graph by its left Cayley graph. And 
Lemma \ref{lemma.length property}(ii) ensures that left-right symmetry is preserved at all levels of the proofs, so we 
could replicate all the preceding proofs using $\mathsf{LCay}(S,A)$ and 
$\mathsf{KR}^{\mathbbm{1}}_{\operatorname{left}}(S,A)$.
\end{remark}

In the next paper we will expand this theory and apply it to mixing times.

%%%%%%%%%%%%%%%%%%%%%%%%%%%%%%%%%%%%%%%%%%%%%%%%%%%%%%%%
\section{Examples}
\label{section.examples}
%%%%%%%%%%%%%%%%%%%%%%%%%%%%%%%%%%%%%%%%%%%%%%%%%%%%%%%%

%%%%%%%%%%%%%%%%%%%%%%%%%%%%%%%%%%%%%%%%%%%%%%%%%%%%%%%%
\subsection{Left action on semaphore codes}

Let $A$ be a finite alphabet and let $k \geqslant 1$. Consider the left action $A^+ \times A^k \to A^k$ defined as follows: 
given $v \in A^+$ and $u \in A^k$, let $v\cdot u$ denote the prefix of length $k$ of $vu$. An equivalence relation $\rho$ on $A^k$ is 
a \defn{left congruence} if
\[
	u\rho v \Rightarrow (w \cdot u)\rho(w \cdot v)
\]
holds for all $u,v \in A^k$ and $w \in A^+$. Then the set $\rho\backslash A^k$ of all $\rho$-classes becomes a left zero 
semigroup under the operation $(\rho u)(\rho v) = \rho u$. We
denote by $\operatorname{LC}(A^k)$ the set of left congruences on $A^k$.
As shown in~\cite{RSS.2016a,RSS.2016b}, every left congruence can be approximated by a special left congruence and
special left congruences are in bijection with semaphore codes. A \defn{semaphore code} \cite{BPR.2010} is a prefix code 
$\mathcal{S}$ over $A$ (i.e., all elements in the code are incomparable in prefix order) for which there is a left action in the following sense:
If $a\in A$ and $u \in \mathcal{S}$, then $au$ has a prefix in $\mathcal{S}$. The left action $a\cdot u$ is the prefix 
of $au$ that is in $\mathcal{S}$. 

\begin{remark}
In fact, \cite{RSS.2016a,RSS.2016b} use right congruences and semaphore codes are suffix codes. But as outlined in the introduction,
in applications the left action is usually used. See also~\cite{RhodesSchilling.2019}.
\end{remark}

Given a semaphore code $\mathcal{S}$ and the left action by $A^+$ on $\mathcal{S}$, consider the functions $\mathcal{S} \to \mathcal{S}$
induced by this action. This yields a semigroup $(S,A)$. The right Cayley graph of $(S,A)$ is equal to its Karnofsky--Rhodes expansion.

\begin{example}
\label{example.KRleft}
Let $A=\{a,b\}$ be a two letter alphabet and $I$ the ideal in $A^*$ generated by $aaa,aab,aba,baa,bab$. 
Then the left action prefix semaphore code in $A^4$ is given in Table~\ref{table.sc}. Hence $S$ has 11 elements.
\begin{table}
\[
\begin{array}{l|ll}
 & a\cdot & b\cdot \\ \hline
 aaa & aaa & baa\\
 aab & aaa & baa\\
 aba & aab & bab\\
 baa & aba & bbaa\\
 bab & aba & bbab\\
 bbaa & abba & bbba\\
 abba & aab & bab\\
 bbba & abbb & bbbb\\
 bbab & abba & bbba\\
 abbb & aab & bab\\
 bbbb & abbb & bbbb\\
 \hline
\end{array}
\]
\caption{The left action semaphore code in $\{a,b\}^4$ associated to the ideal generated by $aaa,aab,aba,baa,bab$.
\label{table.sc}}
\end{table}
To compute $\mathsf{KR}_{\operatorname{right}}(S,A)$, we compute the action of the various subwords of the elements
in $S$ on $\mathcal{S}$ and record the images, see Figure~\ref{figure.KRleft}. We have $\ell=8$, so the Chiswell tree has 9 levels. However,
the Lyndon--Chiswell length function $D$ cannot take on the values $2,4,6,8$. It follows that these levels are equal to their predecessors 
and can be omitted. The elliptic left action of $a$ on the Chiswell tree is given in Figure~\ref{figure.KRdota}, whereas the elliptic left
action of $b$ is given in Figure~\ref{figure.KRdotb}.
\begin{figure}[h!]
\scalebox{0.75}{
\begin{tikzpicture}[auto]
\node (I) at (0, 0) {$\mathbbm{1}$};
\node (A) at (-3,-1.5) {$\{aaa,aba,aab,abba,abbb\}$};
\node(B) at (3,-1.5) {$\{baa,bab,bbaa,bbba,bbab,bbbb\}$};
\node(C) at (-5,-3) {$\{aaa,aab\}$};
\node(D) at (-2,-3) {$\{aba,abba,abbb\}$};
\node(E) at (2,-3) {$\{baa,bab\}$};
\node(F) at (5,-3) {$\{bbaa,bbba,bbab,bbbb\}$};
\node(G) at (-5.7,-4.5) {$\{aaa\}$};
\node(H) at (-4.3,-4.5) {$\{aab\}$};
\node(J) at (-2.8,-4.5) {$\{aba\}$};
\node(K) at (-1.3,-4.5) {$\{abba,abbb\}$};
\node(L) at (1.3,-4.5) {$\{baa\}$};
\node(M) at (2.8,-4.5) {$\{bab\}$};
\node(N) at (4.3,-4.5) {$\{bbaa,bbab\}$};
\node(O) at (6.3,-4.5) {$\{bbba,bbbb\}$};
\node(P) at (-2,-6) {$\{abba\}$};
\node(Q) at (-0.3,-6) {$\{abbb\}$};
\node(R) at (3.7,-6) {$\{bbaa\}$};
\node(S) at (4.9,-6) {$\{bbab\}$};
\node(T) at (5.9,-6) {$\{bbba\}$};
\node(U) at (6.9,-6) {$\{bbbb\}$};

\draw[edge,blue,thick] (I) -- (A) node[midway, left] {$a$\;\;};
\draw[edge,blue,thick] (I) -- (B) node[midway, right] {\;\;$b$};
\draw[edge,blue,thick] (A) -- (C) node[midway, left] {$a$\;\;};
\draw[edge,blue,thick] (A) -- (D) node[midway, right] {\;\;$b$};
\draw[edge,blue,thick] (B) -- (E) node[midway, left] {$a$\;\;};
\draw[edge,blue,thick] (B) -- (F) node[midway, right] {\;\;$b$};
\draw[edge,blue,thick] (C) -- (G) node[midway, left] {$a$\;\;};
\draw[edge,blue,thick] (C) -- (H) node[midway, right] {\;\;$b$};
\draw[edge,blue,thick] (D) -- (J) node[midway, left] {$a$\;\;};
\draw[edge,blue,thick] (D) -- (K) node[midway, right] {\;\;$b$};
\draw[edge,blue,thick] (E) -- (L) node[midway, left] {$a$\;\;};
\draw[edge,blue,thick] (E) -- (M) node[midway, right] {\;\;$b$};
\draw[edge,blue,thick] (F) -- (N) node[midway, left] {$a$\;\;};
\draw[edge,blue,thick] (F) -- (O) node[midway, right] {\;\;$b$};
\draw[edge,blue,thick] (K) -- (P) node[midway, left] {$a$\;\;};
\draw[edge,blue,thick] (K) -- (Q) node[midway, right] {\;\;$b$};
\draw[edge,blue,thick] (N) -- (R) node[midway, left] {$a$\;\;};
\draw[edge,blue,thick] (N) -- (S) node[midway, right] {\;\;$b$};
\draw[edge,blue,thick] (O) -- (T) node[midway, left] {$a$\;\;};
\draw[edge,blue,thick] (O) -- (U) node[midway, right] {\;\;$b$};
\end{tikzpicture}}
\caption{The right Karnofsky--Rhodes expansion of  $\mathsf{RCay}(S,A)$ of Example~\ref{example.KRleft}, where the action on the leaves is omitted.
\label{figure.KRleft}}
\end{figure}

\begin{figure}[h!]
\scalebox{0.75}{
\begin{tikzpicture}[auto]
\node (I) at (0, 0) {$\mathbbm{1}$};
\node (A) at (-3,-1.5) {$\{aaa,aba,aab,abba,abbb\}$};
\node(B) at (3,-1.5) {$\{baa,bab,bbaa,bbba,bbab,bbbb\}$};
\node(C) at (-5,-3) {$\{aaa,aab\}$};
\node(D) at (-2,-3) {$\{aba,abba,abbb\}$};
\node(E) at (2,-3) {$\{baa,bab\}$};
\node(F) at (5,-3) {$\{bbaa,bbba,bbab,bbbb\}$};
\node(G) at (-5.7,-4.5) {$\{aaa\}$};
\node(H) at (-4.3,-4.5) {$\{aab\}$};
\node(J) at (-2.8,-4.5) {$\{aba\}$};
\node(K) at (-1.3,-4.5) {$\{abba,abbb\}$};
\node(L) at (1.3,-4.5) {$\{baa\}$};
\node(M) at (2.8,-4.5) {$\{bab\}$};
\node(N) at (4.3,-4.5) {$\{bbaa,bbab\}$};
\node(O) at (6.3,-4.5) {$\{bbba,bbbb\}$};
\node(P) at (-1.8,-6) {$\{abba\}$};
\node(Q) at (-0.3,-6) {$\{abbb\}$};
\node(R) at (3.8,-6) {$\{bbaa\}$};
\node(S) at (4.9,-6) {$\{bbab\}$};
\node(T) at (5.9,-6) {$\{bbba\}$};
\node(U) at (6.9,-6) {$\{bbbb\}$};
\node(V) at (-5.7,-6) {$\{aaa\}$};
\node(W) at (-4.3,-6) {$\{aab\}$};
\node(X) at (-2.8,-6) {$\{aba\}$};
\node(Y) at (1.3,-6) {$\{baa\}$};
\node(Z) at (2.8,-6) {$\{bab\}$};

%\draw[edge,blue,thick] (I) -- (A) node[midway, left] {$a$\;\;};
%\draw[edge,blue,thick] (I) -- (B) node[midway, right] {\;\;$b$};
%\draw[edge,blue,thick] (A) -- (C) node[midway, left] {$a$\;\;};
%\draw[edge,blue,thick] (A) -- (D) node[midway, right] {\;\;$b$};
%\draw[edge,blue,thick] (B) -- (E) node[midway, left] {$a$\;\;};
%\draw[edge,blue,thick] (B) -- (F) node[midway, right] {\;\;$b$};
%\draw[edge,blue,thick] (C) -- (G) node[midway, left] {$a$\;\;};
%\draw[edge,blue,thick] (C) -- (H) node[midway, right] {\;\;$b$};
%\draw[edge,blue,thick] (D) -- (J) node[midway, left] {$a$\;\;};
%\draw[edge,blue,thick] (D) -- (K) node[midway, right] {\;\;$b$};
%\draw[edge,blue,thick] (E) -- (L) node[midway, left] {$a$\;\;};
%\draw[edge,blue,thick] (E) -- (M) node[midway, right] {\;\;$b$};
%\draw[edge,blue,thick] (F) -- (N) node[midway, left] {$a$\;\;};
%\draw[edge,blue,thick] (F) -- (O) node[midway, right] {\;\;$b$};
%\draw[edge,blue,thick] (K) -- (P) node[midway, left] {$a$\;\;};
%\draw[edge,blue,thick] (K) -- (Q) node[midway, right] {\;\;$b$};
%\draw[edge,blue,thick] (N) -- (R) node[midway, left] {$a$\;\;};
%\draw[edge,blue,thick] (N) -- (S) node[midway, right] {\;\;$b$};
%\draw[edge,blue,thick] (O) -- (T) node[midway, left] {$a$\;\;};
%\draw[edge,blue,thick] (O) -- (U) node[midway, right] {\;\;$b$};
\draw[edge,blue,thick] (I) -- (A) node[midway, left] {};
\draw[edge,blue,thick] (I) -- (B) node[midway, right] {};
\draw[edge,blue,thick] (A) -- (C) node[midway, left] {};
\draw[edge,blue,thick] (A) -- (D) node[midway, right] {};
\draw[edge,blue,thick] (B) -- (E) node[midway, left] {};
\draw[edge,blue,thick] (B) -- (F) node[midway, right] {};
\draw[edge,blue,thick] (C) -- (G) node[midway, left] {};
\draw[edge,blue,thick] (C) -- (H) node[midway, right] {};
\draw[edge,blue,thick] (D) -- (J) node[midway, left] {};
\draw[edge,blue,thick] (D) -- (K) node[midway, right] {};
\draw[edge,blue,thick] (E) -- (L) node[midway, left] {};
\draw[edge,blue,thick] (E) -- (M) node[midway, right] {};
\draw[edge,blue,thick] (F) -- (N) node[midway, left] {};
\draw[edge,blue,thick] (F) -- (O) node[midway, right] {};
\draw[edge,blue,thick] (K) -- (P) node[midway, left] {};
\draw[edge,blue,thick] (K) -- (Q) node[midway, right] {};
\draw[edge,blue,thick] (N) -- (R) node[midway, left] {};
\draw[edge,blue,thick] (N) -- (S) node[midway, right] {};
\draw[edge,blue,thick] (O) -- (T) node[midway, left] {};
\draw[edge,blue,thick] (O) -- (U) node[midway, right] {};

\draw[edge,blue,thick] (G) -- (V);
\draw[edge,blue,thick] (H) -- (W);
\draw[edge,blue,thick] (J) -- (X);
\draw[edge,blue,thick] (L) -- (Y);
\draw[edge,blue,thick] (M) -- (Z);

\path (B) edge[->,thick, red, bend right=15]  (A);
\path (A) edge[->,loop left, thick, red]  (A);

\path (D) edge[->,thick, red, bend right=15]  (C);
\path (E) edge[->,thick, red, bend right=15]  (D);
\path (F) edge[->,thick, red, bend right=15]  (D);
\path (C) edge[->,loop left, thick, red]  (C);

\path (H) edge[->,thick, red, bend right=30]  (G);
\path (J) edge[->,thick, red, bend right=15]  (H);
\path (K) edge[->,thick, red, bend right=15]  (H);
\path (L) edge[->,thick, red, bend right=15]  (J);
\path (M) edge[->,thick, red, bend right=15]  (J);
\path (N) edge[->,thick, red, bend right=15]  (K);
\path (O) edge[->,thick, red, bend right=15]  (K);
\path (G) edge[->,loop left, thick, red]  (G);

\path (U) edge[->,thick, red, bend right=20]  (Q);
\path (T) edge[->,thick, red, bend right=20]  (Q);
\path (S) edge[->,thick, red, bend right=20]  (P);
\path (R) edge[->,thick, red, bend right=20]  (P);
\path (Z) edge[->,thick, red, bend right=20]  (X);
\path (Y) edge[->,thick, red, bend right=20]  (X);
\path (Q) edge[->,thick, red, bend right=30]  (W);
\path (P) edge[->,thick, red, bend right=30]  (W);
\path (X) edge[->,thick, red, bend right=30]  (W);
\path (W) edge[->,thick, red, bend right=30]  (V);
\path (V) edge[->,loop left, thick, red]  (V);
\end{tikzpicture}}
\caption{The action of {\color{red}$a\cdot$} on the semaphore code induces the action level-by-level on the the Chiswell tree 
of Example~\ref{example.KRleft} (in red).
\label{figure.KRdota}}
\end{figure}

\begin{figure}[h!]
\scalebox{0.75}{
\begin{tikzpicture}[auto]
\node (I) at (0, 0) {$\mathbbm{1}$};
\node (A) at (-3,-1.5) {$\{aaa,aba,aab,abba,abbb\}$};
\node(B) at (3,-1.5) {$\{baa,bab,bbaa,bbba,bbab,bbbb\}$};
\node(C) at (-5,-3) {$\{aaa,aab\}$};
\node(D) at (-2,-3) {$\{aba,abba,abbb\}$};
\node(E) at (2,-3) {$\{baa,bab\}$};
\node(F) at (5,-3) {$\{bbaa,bbba,bbab,bbbb\}$};
\node(G) at (-5.7,-4.5) {$\{aaa\}$};
\node(H) at (-4.3,-4.5) {$\{aab\}$};
\node(J) at (-2.8,-4.5) {$\{aba\}$};
\node(K) at (-1.3,-4.5) {$\{abba,abbb\}$};
\node(L) at (1.3,-4.5) {$\{baa\}$};
\node(M) at (2.8,-4.5) {$\{bab\}$};
\node(N) at (4.3,-4.5) {$\{bbaa,bbab\}$};
\node(O) at (6.3,-4.5) {$\{bbba,bbbb\}$};
\node(P) at (-1.8,-6) {$\{abba\}$};
\node(Q) at (-0.3,-6) {$\{abbb\}$};
\node(R) at (3.8,-6) {$\{bbaa\}$};
\node(S) at (4.9,-6) {$\{bbab\}$};
\node(T) at (5.9,-6) {$\{bbba\}$};
\node(U) at (6.9,-6) {$\{bbbb\}$};
\node(V) at (-5.7,-6) {$\{aaa\}$};
\node(W) at (-4.3,-6) {$\{aab\}$};
\node(X) at (-2.8,-6) {$\{aba\}$};
\node(Y) at (1.3,-6) {$\{baa\}$};
\node(Z) at (2.8,-6) {$\{bab\}$};

\draw[edge,blue,thick] (I) -- (A) node[midway, left] {};
\draw[edge,blue,thick] (I) -- (B) node[midway, right] {};
\draw[edge,blue,thick] (A) -- (C) node[midway, left] {};
\draw[edge,blue,thick] (A) -- (D) node[midway, right] {};
\draw[edge,blue,thick] (B) -- (E) node[midway, left] {};
\draw[edge,blue,thick] (B) -- (F) node[midway, right] {};
\draw[edge,blue,thick] (C) -- (G) node[midway, left] {};
\draw[edge,blue,thick] (C) -- (H) node[midway, right] {};
\draw[edge,blue,thick] (D) -- (J) node[midway, left] {};
\draw[edge,blue,thick] (D) -- (K) node[midway, right] {};
\draw[edge,blue,thick] (E) -- (L) node[midway, left] {};
\draw[edge,blue,thick] (E) -- (M) node[midway, right] {};
\draw[edge,blue,thick] (F) -- (N) node[midway, left] {};
\draw[edge,blue,thick] (F) -- (O) node[midway, right] {};
\draw[edge,blue,thick] (K) -- (P) node[midway, left] {};
\draw[edge,blue,thick] (K) -- (Q) node[midway, right] {};
\draw[edge,blue,thick] (N) -- (R) node[midway, left] {};
\draw[edge,blue,thick] (N) -- (S) node[midway, right] {};
\draw[edge,blue,thick] (O) -- (T) node[midway, left] {};
\draw[edge,blue,thick] (O) -- (U) node[midway, right] {};
%\draw[edge,blue,thick] (I) -- (A) node[midway, left] {$a$\;\;};
%\draw[edge,blue,thick] (I) -- (B) node[midway, right] {\;\;$b$};
%\draw[edge,blue,thick] (A) -- (C) node[midway, left] {$a$\;\;};
%\draw[edge,blue,thick] (A) -- (D) node[midway, right] {\;\;$b$};
%\draw[edge,blue,thick] (B) -- (E) node[midway, left] {$a$\;\;};
%\draw[edge,blue,thick] (B) -- (F) node[midway, right] {\;\;$b$};
%\draw[edge,blue,thick] (C) -- (G) node[midway, left] {$a$\;\;};
%\draw[edge,blue,thick] (C) -- (H) node[midway, right] {\;\;$b$};
%\draw[edge,blue,thick] (D) -- (J) node[midway, left] {$a$\;\;};
%\draw[edge,blue,thick] (D) -- (K) node[midway, right] {\;\;$b$};
%\draw[edge,blue,thick] (E) -- (L) node[midway, left] {$a$\;\;};
%\draw[edge,blue,thick] (E) -- (M) node[midway, right] {\;\;$b$};
%\draw[edge,blue,thick] (F) -- (N) node[midway, left] {$a$\;\;};
%\draw[edge,blue,thick] (F) -- (O) node[midway, right] {\;\;$b$};
%\draw[edge,blue,thick] (K) -- (P) node[midway, left] {$a$\;\;};
%\draw[edge,blue,thick] (K) -- (Q) node[midway, right] {\;\;$b$};
%\draw[edge,blue,thick] (N) -- (R) node[midway, left] {$a$\;\;};
%\draw[edge,blue,thick] (N) -- (S) node[midway, right] {\;\;$b$};
%\draw[edge,blue,thick] (O) -- (T) node[midway, left] {$a$\;\;};
%\draw[edge,blue,thick] (O) -- (U) node[midway, right] {\;\;$b$};
\draw[edge,blue,thick] (G) -- (V);
\draw[edge,blue,thick] (H) -- (W);
\draw[edge,blue,thick] (J) -- (X);
\draw[edge,blue,thick] (L) -- (Y);
\draw[edge,blue,thick] (M) -- (Z);

\path (A) edge[->,thick, darkgreen, bend left=15]  (B);
\path (B) edge[->,loop right, thick, darkgreen]  (B);

\path (C) edge[->,thick, darkgreen, bend left=15]  (E);
\path (D) edge[->,thick, darkgreen, bend left=15]  (E);
\path (E) edge[->,thick, darkgreen, bend left=15]  (F);
\path (F) edge[->,loop right, thick, darkgreen]  (F);

\path (G) edge[->,thick, darkgreen, bend left=20]  (L);
\path (H) edge[->,thick, darkgreen, bend left=20]  (L);
\path (J) edge[->,thick, darkgreen, bend left=15]  (M);
\path (K) edge[->,thick, darkgreen, bend left=15]  (M);
\path (L) edge[->,thick, darkgreen, bend left=30]  (N);
\path (M) edge[->,thick, darkgreen, bend left=30]  (N);
\path (N) edge[->,thick, darkgreen, bend left=30]  (O);
\path (O) edge[->,loop right, thick, darkgreen]  (O);

\path (V) edge[->,thick, darkgreen, bend left=20]  (Y);
\path (W) edge[->,thick, darkgreen, bend left=20]  (Y);
\path (X) edge[->,thick, darkgreen, bend left=20]  (Z);
\path (P) edge[->,thick, darkgreen, bend left=20]  (Z);
\path (Q) edge[->,thick, darkgreen, bend left=20]  (Z);
\path (Y) edge[->,thick, darkgreen, bend left=30]  (R);
\path (Z) edge[->,thick, darkgreen, bend left=60]  (R);
\path (R) edge[->,thick, darkgreen, bend left=40]  (T);

\path (S) edge[->,thick, darkgreen, bend left=60]  (T);
\path (T) edge[->,thick, darkgreen, bend left=60]  (U);
\path (U) edge[->,loop right, thick, darkgreen]  (U);
\end{tikzpicture}}
\caption{The action of {\color{darkgreen}$b\cdot$} on the semaphore code induces the action level-by-level 
on the the Chiswell tree of Example~\ref{example.KRleft} (in green).
\label{figure.KRdotb}}
\end{figure}
\end{example}

%%%%%%%%%%%%%%%%%%%%%%%%%%%%%%%%%%%%%%%%%%%%%%%%%%%
\subsection{Right zero semigroup with two generators}

Let $(S,A)$ be the right zero semigroup $\mathsf{RZ}(2)$ (that is $xy=y$ for all $x,y\in \mathsf{RZ}(2)$)
with two generators $A=\{a,b\}$. The Karnofsky--Rhodes expansion of the right Cayley graph of $(S,A)$
is depicted in Figure~\ref{figure.RZ}. Then the Chiswell construction is given in Figure~\ref{figure.Chiswell RZ}.
The right and left actions of $a$ and $b$ on the Chiswell construction are given in 
Figures~\ref{figure.Chiswell RZa},~\ref{figure.Chiswell RZb},~\ref{figure.Chiswell RZaL}~and~\ref{figure.Chiswell RZbL}, 
respectively.

\begin{figure}[h!]
\begin{center}
\begin{tikzpicture}
\node (A) at (0, 0) {$\mathbbm{1}$};
\node (B) at (-1.5,-1) {$a$};
\node(C) at (1.5,-1) {$b$};
\node (D) at (-1.5,-2.5) {$b$};
\node (E) at (1.5,-2.5) {$a$};

\draw[edge,blue,thick] (A) -- (B) node[midway, above] {$a$};
\draw[edge,thick,blue] (A) -- (C) node[midway, above] {$b$};
\path (B) edge[->,thick, bend right = 20] node[midway,left] {$b$} (D);
\path (D) edge[->,thick, bend right = 20] node[midway,right] {$a$} (B);
\path (B) edge[->,loop left, thick]  node[midway,left]{$a$} (B);
\path (D) edge[->,loop left, thick]  node[midway,left]{$b$} (D);
\path (C) edge[->,thick, bend right = 20] node[midway,left] {$a$} (E);
\path (E) edge[->,thick, bend right = 20] node[midway,right] {$b$} (C);
\path (C) edge[->,loop right, thick]  node[midway,right]{$b$} (C);
\path (E) edge[->,loop right, thick]  node[midway,right]{$a$} (D);

\end{tikzpicture}
\end{center}
\caption{$\mathsf{KR}_{\operatorname{right}}(\mathsf{RZ}(2),\{a,b\})$.
\label{figure.RZ}
} 
\end{figure}

\begin{figure}[h!]
\begin{center}
\begin{tikzpicture}
\node (A) at (0, 0) {$[0,\mathbbm{1}]$};
\node (B) at (-2, -1.5) {$[1,ab]=[1,a]$};
\node (C) at (2, -1.5) {$[1,b]=[1,ba]$};
\node (D) at (-3, -3) {$[2,a]$};
\node (E) at (-1, -3) {$[2,ab]$};
\node (F) at (1, -3) {$[2,b]$};
\node (G) at (3, -3) {$[2,ba]$};
\node (H) at (-6, -1.5) {$[1,\mathbbm{1}]$};
\node (I) at (-6, -3) {$[2,\mathbbm{1}]$};

\draw[edge,blue,thick] (A) -- (B) node[midway, above] {$a$};
\draw[edge,thick,blue] (A) -- (C) node[midway, above] {$b$};
\draw[edge,thick,blue] (B) -- (D);
\draw[edge,thick,blue] (B) -- (E);
\draw[edge,thick,blue] (C) -- (F);
\draw[edge,thick,blue] (C) -- (G);
\draw[edge,blue,thick] (A) -- (H);
\draw[edge,blue,thick] (H) -- (I);

\end{tikzpicture}
\end{center}
\caption{The Chiswell construction for $\mathsf{KR}_{\operatorname{right}}(\mathsf{RZ}(2),\{a,b\})$.
\label{figure.Chiswell RZ}
} 
\end{figure}

\newpage

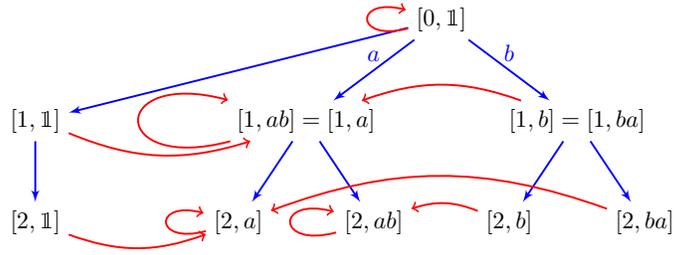
\begin{figure}
\begin{center}
\scalebox{0.9}{
\begin{tikzpicture}
\node (A) at (0, 0) {$[0,\mathbbm{1}]$};
\node (B) at (-2, -1.5) {$[1,ab]=[1,a]$};
\node (C) at (2, -1.5) {$[1,b]=[1,ba]$};
\node (D) at (-3, -3) {$[2,a]$};
\node (E) at (-1, -3) {$[2,ab]$};
\node (F) at (1, -3) {$[2,b]$};
\node (G) at (3, -3) {$[2,ba]$};
\node (H) at (-6, -1.5) {$[1,\mathbbm{1}]$};
\node (I) at (-6, -3) {$[2,\mathbbm{1}]$};

\draw[edge,blue,thick] (A) -- (B) node[midway, above] {$a$};
\draw[edge,thick,blue] (A) -- (C) node[midway, above] {$b$};
\draw[edge,thick,blue] (B) -- (D);
\draw[edge,thick,blue] (B) -- (E);
\draw[edge,thick,blue] (C) -- (F);
\draw[edge,thick,blue] (C) -- (G);
\draw[edge,blue,thick] (A) -- (H);
\draw[edge,blue,thick] (H) -- (I);

\path (A) edge[->,loop left, thick, red]  (A);
\path (C) edge[->,thick, red, bend right=20]  (B);
\path (B) edge[->,loop left, thick, red]  (B);
\path (F) edge[->,thick, red, bend right=20]  (E);
\path (G) edge[->,thick, red, bend right=20]  (D);
\path (D) edge[->,loop left, thick, red]  (D);
\path (E) edge[->,loop left, thick, red]  (E);
\path (H) edge[->,thick, red, bend right=20]  (B);
\path (I) edge[->,thick, red, bend right=20]  (D);

\end{tikzpicture}}
\end{center}
\caption{The left action of {\color{red}$a\cdot $} on the Chiswell tree for 
$\mathsf{KR}_{\operatorname{right}}(\mathsf{RZ}(2),\{a,b\})$.
\label{figure.Chiswell RZa}
} 
\end{figure}

\begin{figure}
\begin{center}
\scalebox{0.9}{
\begin{tikzpicture}
\node (A) at (0, 0) {$[0,\mathbbm{1}]$};
\node (B) at (-2, -1.5) {$[1,ab]=[1,a]$};
\node (C) at (2, -1.5) {$[1,b]=[1,ba]$};
\node (D) at (-3, -3) {$[2,a]$};
\node (E) at (-1, -3) {$[2,ab]$};
\node (F) at (1, -3) {$[2,b]$};
\node (G) at (3, -3) {$[2,ba]$};
\node (H) at (-6, -1.5) {$[1,\mathbbm{1}]$};
\node (I) at (-6, -3) {$[2,\mathbbm{1}]$};

\draw[edge,blue,thick] (A) -- (B) node[midway, above] {$a$};
\draw[edge,thick,blue] (A) -- (C) node[midway, above] {$b$};
\draw[edge,thick,blue] (B) -- (D);
\draw[edge,thick,blue] (B) -- (E);
\draw[edge,thick,blue] (C) -- (F);
\draw[edge,thick,blue] (C) -- (G);
\draw[edge,blue,thick] (A) -- (H);
\draw[edge,blue,thick] (H) -- (I);

\path (A) edge[->,loop left, thick, darkgreen]  (A);
\path (B) edge[->,thick, darkgreen, bend left=20]  (C);
\path (C) edge[->,loop right, thick, darkgreen]  (C);
\path (E) edge[->,thick, darkgreen, bend left=20]  (F);
\path (D) edge[->,thick, darkgreen, bend left=20]  (G);
\path (F) edge[->,loop right, thick, darkgreen]  (F);
\path (G) edge[->,loop right, thick, darkgreen]  (G);
\path (H) edge[->,thick, darkgreen, bend right=20]  (C);
\path (I) edge[->,thick, darkgreen, bend right=20]  (F);

\end{tikzpicture}}
\end{center}
\caption{The left action of {\color{darkgreen}$b\cdot$} on the Chiswell tree for 
$\mathsf{KR}_{\operatorname{right}}(\mathsf{RZ}(2),\{a,b\})$.
\label{figure.Chiswell RZb}
} 
\end{figure}
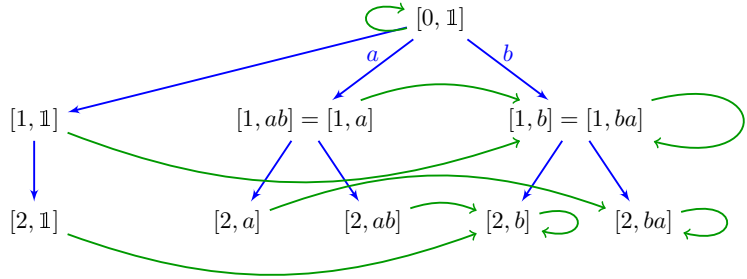

\begin{figure}
\begin{center}
\scalebox{0.9}{
\begin{tikzpicture}
\node (A) at (0, 0) {$[0,\mathbbm{1}]$};
\node (B) at (-2, -1.5) {$[1,ab]=[1,a]$};
\node (C) at (2, -1.5) {$[1,b]=[1,ba]$};
\node (D) at (-3, -3) {$[2,a]$};
\node (E) at (-1, -3) {$[2,ab]$};
\node (F) at (1, -3) {$[2,b]$};
\node (G) at (3, -3) {$[2,ba]$};
\node (H) at (-6, -1.5) {$[1,\mathbbm{1}]$};
\node (I) at (-6, -3) {$[2,\mathbbm{1}]$};

\draw[edge,blue,thick] (A) -- (B) node[midway, above] {$a$};
\draw[edge,thick,blue] (A) -- (C) node[midway, above] {$b$};
\draw[edge,thick,blue] (B) -- (D);
\draw[edge,thick,blue] (B) -- (E);
\draw[edge,thick,blue] (C) -- (F);
\draw[edge,thick,blue] (C) -- (G);
\draw[edge,blue,thick] (A) -- (H);
\draw[edge,blue,thick] (H) -- (I);

\path (A) edge[->,loop left, thick, red]  (A);
\path (C) edge[->,loop right, thick, red]  (C);
\path (B) edge[->,loop left, thick, red]  (B);
\path (F) edge[->,thick, red, bend right=20]  (G);
\path (G) edge[->,loop right, thick, red]  (G);
\path (D) edge[->,loop left, thick, red]  (D);
\path (E) edge[->,thick, red, bend left=20]  (D);
\path (H) edge[->,thick, red, bend right=20]  (B);
\path (I) edge[->,thick, red, bend right=20]  (D);

\end{tikzpicture}}
\end{center}
\caption{The right action of {\color{red}$\cdot a$} on the Chiswell tree for 
$\mathsf{KR}_{\operatorname{right}}(\mathsf{RZ}(2),\{a,b\})$.
\label{figure.Chiswell RZaL}
} 
\end{figure}
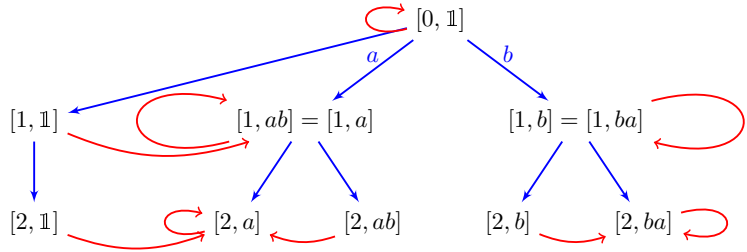

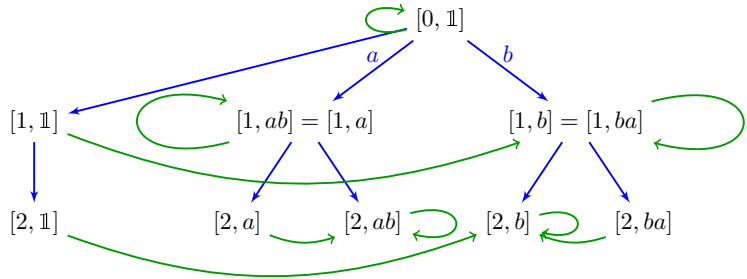
\begin{figure}
\begin{center}
\scalebox{0.9}{
\begin{tikzpicture}
\node (A) at (0, 0) {$[0,\mathbbm{1}]$};
\node (B) at (-2, -1.5) {$[1,ab]=[1,a]$};
\node (C) at (2, -1.5) {$[1,b]=[1,ba]$};
\node (D) at (-3, -3) {$[2,a]$};
\node (E) at (-1, -3) {$[2,ab]$};
\node (F) at (1, -3) {$[2,b]$};
\node (G) at (3, -3) {$[2,ba]$};
\node (H) at (-6, -1.5) {$[1,\mathbbm{1}]$};
\node (I) at (-6, -3) {$[2,\mathbbm{1}]$};

\draw[edge,blue,thick] (A) -- (B) node[midway, above] {$a$};
\draw[edge,thick,blue] (A) -- (C) node[midway, above] {$b$};
\draw[edge,thick,blue] (B) -- (D);
\draw[edge,thick,blue] (B) -- (E);
\draw[edge,thick,blue] (C) -- (F);
\draw[edge,thick,blue] (C) -- (G);
\draw[edge,blue,thick] (A) -- (H);
\draw[edge,blue,thick] (H) -- (I);

\path (A) edge[->,loop left, thick, darkgreen]  (A);
\path (B) edge[->,loop left, thick, darkgreen]  (B);
\path (C) edge[->,loop right, thick, darkgreen]  (C);
\path (D) edge[->,thick, darkgreen, bend right=20]  (E);
\path (E) edge[->,loop right, thick, darkgreen]  (E);
\path (F) edge[->,loop right, thick, darkgreen]  (F);
\path (G) edge[->,thick, darkgreen, bend left=20]  (F);
\path (H) edge[->,thick, darkgreen, bend right=20]  (C);
\path (I) edge[->,thick, darkgreen, bend right=20]  (F);

\end{tikzpicture}}
\end{center}
\caption{The right action of {\color{darkgreen}$\cdot b$} on the Chiswell tree for 
$\mathsf{KR}_{\operatorname{right}}(\mathsf{RZ}(2),\{a,b\})$.
\label{figure.Chiswell RZbL}
} 
\end{figure}

\newpage

%%%%%%%%%%%%%%%%%%%%%%%%%%%%%%%%%%%%%%%%%%%%%%%%%%%
\subsection{The monoid $T_2$ of total transformations}

Let $T_2$ denote the monoid of total transformations on the set $\{ 1,2\}$. We denote $\varphi \in T_2$ by 
$(\varphi 1\; \varphi 2)$ (so in $\psi \varphi$ the map $\varphi$ acts first). Let $A = \{ a,b\}$ and let 
$\varphi \colon A^* \to S^{\mathbbm{1}}$ be the monoid homomorphism defined by $\varphi(a) = (2\; 1)$ and  
$\varphi(b) = (1\; 1)$. It is routine to check that $\varphi$ is onto and $\mathsf{RCay}(T_2,A) = \mathsf{LCay}(T_2^{\rm op},A)$ 
is depicted in Figure~\ref{figure.T21}. The Karnofsky--Rhodes expansion is given in Figure~\ref{figure.T22},
the Chiswell construction is drawn in Figure~\ref{figure.Chiswell T2}, and
the left action of $a$ on the Chiswell tree is depicted in Figure~\ref{figure.Chiswell T2a}.

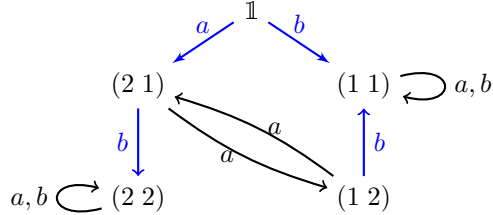
\begin{figure}[h!]
\begin{center}
\begin{tikzpicture}
\node (A) at (0, 0) {$\mathbbm{1}$};
\node (B) at (-1.5,-1) {$(2\; 1)$};
\node(C) at (1.5,-1) {$(1\; 1)$};
\node (D) at (-1.5,-2.5) {$(2\; 2)$};
\node (E) at (1.5,-2.5) {$(1\; 2)$};

\draw[edge,blue,thick] (A) -- (B) node[midway, above] {$a$};
\draw[edge,thick,blue] (A) -- (C) node[midway, above] {$b$};
\path (B) edge[->,thick,blue] node[midway,left] {$b$} (D);
\path (D) edge[->,loop left, thick]  node[midway,left]{$a,b$} (D);
\path (B) edge[->,thick, bend right = 10] node[midway,left] {$a$} (E);
\path (C) edge[->,loop right, thick]  node[midway,right]{$a,b$} (C);
\path (E) edge[->,thick, blue] node[midway,right] {$b$} (C);
\path (E) edge[->,thick, bend right = 10] node[midway,right] {$a$} (B);

\end{tikzpicture}
\end{center}
\caption{$\mathsf{RCay}(T_2,A)$ with the transition edges in blue.
\label{figure.T21}
} 
\end{figure}

\begin{figure}[h!]
\begin{center}
\begin{tikzpicture}
\node (A) at (0, 0) {$\mathbbm{1}$};
\node (B) at (-1.5,-1) {$a$};
\node(C) at (1.5,-1) {$b$};
\node (D) at (-1.5,-2.5) {$ab$};
\node (E) at (1.5,-2.5) {$a^2$};
\node (F) at (1.5,-4) {$a^2b$};

\draw[edge,blue,thick] (A) -- (B) node[midway, above] {$a$};
\draw[edge,thick,blue] (A) -- (C) node[midway, above] {$b$};
\path (B) edge[->,thick,blue] node[midway,left] {$b$} (D);
\path (D) edge[->,loop left, thick]  node[midway,left]{$a,b$} (D);
\path (B) edge[->,thick, bend right = 10] node[midway,left] {$a$} (E);
\path (C) edge[->,loop right, thick]  node[midway,right]{$a,b$} (C);
\path (E) edge[->,thick, blue] node[midway,right] {$b$} (F);
\path (E) edge[->,thick, bend right = 10] node[midway,right] {$a$} (B);
\path (F) edge[->,loop right, thick]  node[midway,right]{$a,b$} (F);

\end{tikzpicture}
\end{center}
\caption{$\mathsf{KR}_{\operatorname{right}}(T_2,A)$.
\label{figure.T22}
} 
\end{figure}

\begin{figure}[t]
\begin{center}
\begin{tikzpicture}
\node (A) at (0, 0) {$[0,\mathbbm{1}]$};
\node (B) at (-2,-1.5) {$[1,\mathbbm{1}]$};
\node (C) at (0,-1.5) {$[1,a]$};
\node (D) at (2,-1.5) {$[1,b]$};
\node (E) at (-2,-3) {$[2,\mathbbm{1}]$};
\node (F) at (0,-3) {$[2,a]$};
\node (G) at (2,-3) {$[2,b]$};
\node (H) at (-2,-4.5) {$[3,\mathbbm{1}]$};
\node (I) at (0,-4.5) {$[3,a]$};
\node (J) at (4,-4.5) {$[3,b]$};
\node (K) at (-4,-6) {$[4,\mathbbm{1}]$};
\node (L) at (-2,-6) {$[4,a]$};
\node (M) at (0,-6) {$[4,a^2]$};
\node (N) at (2,-6) {$[4,ab]$};
\node (O) at (4,-6) {$[4,a^2b]$};
\node (P) at (6,-6) {$[4,b]$};
\draw[edge,blue,thick] (A) -- (B);
\draw[edge,blue,thick] (A) -- (C) node[midway, right] {$a$};
\draw[edge,thick,blue] (A) -- (D) node[midway, above] {$b$};
\draw[edge,thick,blue] (B) -- (E);
\draw[edge,thick,blue] (D) -- (G);
\draw[edge,thick,blue] (C) -- (F);
\draw[edge,thick,blue] (E) -- (H);
\draw[edge,blue,thick] (F) -- (I);
\draw[edge,blue,thick] (G) -- (J);
\draw[edge,thick,blue] (H) -- (K);
\draw[edge,thick,blue] (I) -- (L);
\draw[edge,thick,blue] (I) -- (M);
\draw[edge,thick,blue] (I) -- (N);
\draw[edge,blue,thick] (I) -- (O);
\draw[edge,blue,thick] (J) -- (P);

\end{tikzpicture}
\end{center}
\caption{The Chiswell construction for $\mathsf{KR}_{\operatorname{right}}(T_2,A)$.
\label{figure.Chiswell T2}
} 
\end{figure}

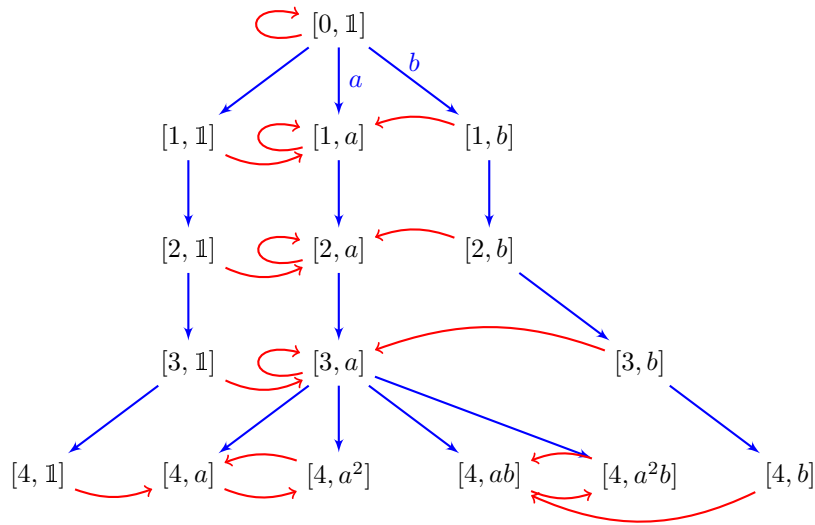
\begin{figure}[t]
\begin{center}
\begin{tikzpicture}
\node (A) at (0, 0) {$[0,\mathbbm{1}]$};
\node (B) at (-2,-1.5) {$[1,\mathbbm{1}]$};
\node (C) at (0,-1.5) {$[1,a]$};
\node (D) at (2,-1.5) {$[1,b]$};
\node (E) at (-2,-3) {$[2,\mathbbm{1}]$};
\node (F) at (0,-3) {$[2,a]$};
\node (G) at (2,-3) {$[2,b]$};
\node (H) at (-2,-4.5) {$[3,\mathbbm{1}]$};
\node (I) at (0,-4.5) {$[3,a]$};
\node (J) at (4,-4.5) {$[3,b]$};
\node (K) at (-4,-6) {$[4,\mathbbm{1}]$};
\node (L) at (-2,-6) {$[4,a]$};
\node (M) at (0,-6) {$[4,a^2]$};
\node (N) at (2,-6) {$[4,ab]$};
\node (O) at (4,-6) {$[4,a^2b]$};
\node (P) at (6,-6) {$[4,b]$};
\draw[edge,blue,thick] (A) -- (B);
\draw[edge,blue,thick] (A) -- (C) node[midway, right] {$a$};
\draw[edge,thick,blue] (A) -- (D) node[midway, above] {$b$};
\draw[edge,thick,blue] (B) -- (E);
\draw[edge,thick,blue] (D) -- (G);
\draw[edge,thick,blue] (C) -- (F);
\draw[edge,thick,blue] (E) -- (H);
\draw[edge,blue,thick] (F) -- (I);
\draw[edge,blue,thick] (G) -- (J);
\draw[edge,thick,blue] (H) -- (K);
\draw[edge,thick,blue] (I) -- (L);
\draw[edge,thick,blue] (I) -- (M);
\draw[edge,thick,blue] (I) -- (N);
\draw[edge,blue,thick] (I) -- (O);
\draw[edge,blue,thick] (J) -- (P);

\path (A) edge[->,loop left, thick, red]  (A);
\path (B) edge[->,thick, red, bend right=25]  (C);
\path (C) edge[->,loop left, thick, red]  (C);
\path (D) edge[->,thick, red, bend right=20]  (C);
\path (E) edge[->,thick, red, bend right=25]  (F);
\path (F) edge[->,loop left, thick, red]  (F);
\path (G) edge[->,thick, red, bend right=20]  (F);
\path (H) edge[->,thick, red, bend right=25]  (I);
\path (I) edge[->,loop left, thick, red]  (I);
\path (J) edge[->,thick, red, bend right=20]  (I);
\path (K) edge[->,thick, red, bend right=20]  (L);
\path (L) edge[->,thick, red, bend right=20]  (M);
\path (M) edge[->,thick, red, bend left=-20]  (L);
\path (N) edge[->,thick, red, bend left=-20]  (O);
\path (O) edge[->,thick, red, bend right=20]  (N);
\path (P) edge[->,thick, red, bend left=25]  (N);

\end{tikzpicture}
\end{center}
\caption{The (left) action of {\color{red}$a\cdot$} on the Chiswell tree for $\mathsf{KR}_{\operatorname{right}}(T_2,A)$.
\label{figure.Chiswell T2a}
} 
\end{figure}

%%%%%%%%%%%%%%%%%%%%%%%%%%%%%%%%%%%%%%%%%%%%%%%%%%%
\bibliographystyle{alpha}
\bibliography{holonomy}{}

\end{document}